\documentclass[11pt,a4paper]{article}

\usepackage{inputenc}
\usepackage{amsmath}
\usepackage{bm}
\usepackage{bbold}
\usepackage{amsthm}
\usepackage{enumerate}

\usepackage[hyphens]{url}

\usepackage{hyperref}
\usepackage{breakurl}

\usepackage{setspace}

\usepackage{algorithm}

\hypersetup{pdfstartview={XYZ null null 1.00},pdfpagemode=UseNone}

\usepackage{pseudocode}

\makeatletter
\renewenvironment{pseudocode}[3][plain]
{%
 \refstepcounter{pseudocode}%
 \ifthenelse{\equal{#1}{plain}}{\setboolean{pcode@plain}{true}}{\setboolean{pcode@plain}{false}}%
 \ifthenelse{\equal{#1}{ruled}}{\setboolean{pcode@ruled}{true}}{\setboolean{pcode@ruled}{false}}%
 \ifthenelse{\equal{#1}{display}}{\setboolean{pcode@disp}{true}}{\setboolean{pcode@disp}{false}}%
 \ifthenelse{\equal{#1}{shadowbox}}{\setboolean{pcode@shad}{true}}{\setboolean{pcode@shad}{false}}%
 \ifthenelse{\equal{#1}{doublebox}}{\setboolean{pcode@dbox}{true}}{\setboolean{pcode@dbox}{false}}%
 \ifthenelse{\equal{#1}{ovalbox}}{\setboolean{pcode@obox}{true}}{\setboolean{pcode@obox}{false}}%
 \ifthenelse{\equal{#1}{Ovalbox}}{\setboolean{pcode@Obox}{true}}{\setboolean{pcode@Obox}{false}}%
 \ifthenelse{\equal{#1}{framebox}}{\setboolean{pcode@fbox}{true}}{\setboolean{pcode@fbox}{false}}%
 \setcounter{pseudonum}{0}%
 \ifthenelse{\boolean{pcode@disp}}%
 {%
  \noindent\begin{math}%
 }%
 {%
  \begin{Sbox}%
	\begin{minipage}{\pcode@width}%
	\ifthenelse{\boolean{pcode@ruled}}
  {
   \noindent\rule{\pcode@width}{1pt}\hfill\\
   {\bfseries Algorithm \thepseudocode:\pcode@tab{1}}\pcode@AF{#2}($#3$)\\
   \noindent\rule{\pcode@width}{1pt}\hfill\\[1ex]
  }
  {
  }
  \noindent\begin{math}\begin{array}{@{\pcode@tab{1}}lr@{}}%
 }{}%
}%
{%
 \ifthenelse{\boolean{pcode@disp}}%
 {%
  \end{math}
 }%
 {%
  \ifthenelse{\boolean{pcode@ruled}}
  {
   \end{array}\end{math}\\[1ex]
   \noindent\rule{\pcode@width}{1pt}\hfill
   \end{minipage}\end{Sbox}\bigskip\noindent%
  }
  {\end{array}\end{math}\end{minipage}\end{Sbox}\smallskip\noindent}%
  \ifthenelse{\boolean{pcode@plain}}{\TheSbox}{}%
  \ifthenelse{\boolean{pcode@ruled}}{\TheSbox}{}%
  \ifthenelse{\boolean{pcode@shad}}{\shadowbox{\TheSbox}}{}%
  \ifthenelse{\boolean{pcode@dbox}}{\doublebox{\TheSbox}}{}%
  \ifthenelse{\boolean{pcode@obox}}{\cornersize*{4ex}\ovalbox{\TheSbox}}{}%
  \ifthenelse{\boolean{pcode@Obox}}{\cornersize*{4ex}\Ovalbox{\TheSbox}}{}%
  \ifthenelse{\boolean{pcode@fbox}}{\fbox{\TheSbox}}{}%
	\smallskip
 }%
}%
\makeatother

\title{Algebraic solution of tropical polynomial optimization problems\thanks{Mathematics, 2021, 9(19), 2472, doi:10.3390/math9192472}}

\author{N. Krivulin\thanks{Faculty of Mathematics and Mechanics, Saint Petersburg State University, 28 Universitetsky Ave., St.~Petersburg, 198504, Russia, 
nkk@math.spbu.ru.}
\thanks{This work was supported in part by the Russian Foundation for Basic Research (grant No. 20-010-00145).}}

\date{}

\newtheorem{theorem}{Theorem}
\newtheorem{lemma}[theorem]{Lemma}

\newtheorem{proposition}[theorem]{Proposition}

\theoremstyle{definition}
\newtheorem{example}{Example}

\begin{document}

\maketitle

\begin{abstract}
We consider constrained optimization problems defined in the tropical algebra setting on a linearly ordered, algebraically complete (radicable) idempotent semifield (a semiring with idempotent addition and invertible multiplication). The problems are to minimize the objective functions given by tropical analogues of multivariate Puiseux polynomials, subject to box constraints on the variables. A technique for variable elimination is presented that converts the original optimization problem to a new one in which one variable is removed and the box constraint for this variable is modified. The novel approach may be thought of as an extension of the Fourier-Motzkin elimination method for systems of linear inequalities in ordered fields to the issue of polynomial optimization in ordered tropical semifields. We use this technique to develop a procedure to solve the problem in a finite number of iterations. The procedure includes two phases: backward elimination and forward substitution of variables. We describe the main steps of the procedure, discuss its computational complexity and present numerical examples.
\\

\textbf{Keywords:} tropical algebra, idempotent semifield, tropical Puiseux polynomial, constrained polynomial optimization problem, box constraint, variable elimination.
\\

\textbf{MSC (2010):} 90C24, 15A80, 90C47, 90C23
\end{abstract}

\section{Introduction}
\section{Introduction}

We consider constrained optimization problems formulated in terms of tropical mathematics~\cite{Baccelli1993Synchronization,Kolokoltsov1997Idempotent,Golan2003Semirings,Itenberg2007Tropical,Gondran2008Graphs,Maclagan2015Introduction}, where the objective functions are defined on a linearly ordered tropical semifield (a semiring with idempotent addition and invertible multiplication). The problems are to minimize tropical analogues of multivariate Puiseux polynomials, subject to box constraints on the variables. A tropical Puiseux polynomial in variables $x_{1},\ldots,x_{N}$ can be defined as a formal analogue of the polynomial in conventional algebra
$$
\sum_{p_{1},\ldots,p_{N}}a_{p_{1},\ldots,p_{N}}x_{1}^{p_{1}}\cdots x_{N}^{p_{N}},
$$
where addition and multiplication (and hence exponentiation) are interpreted in the sense of a tropical semifield, and the exponents $p_{1},\dots,p_{N}$ can be rationals.  

Tropical polynomials have been studied in a range of research contexts, from minimax optimization problems in operations research to tropical algebraic geometry. Specifically, polynomials with non-negative integer exponents over the max-plus real semifield (the max-plus algebra) were studied in~\cite{Cuninghamegreen1980Algebra,Cuninghamegreen1984Using,Zimmermann1992Optimization,Baccelli1993Synchronization,Cuninghamegreen1995Maxpolynomial,Deschutter1996Method,Gondran2008Graphs,Anaola2009Tropical,Hook2015Maxplus}, where addition is defined as taking the maximum and multiplication as arithmetic addition. Several problems are addressed, including polynomial factorization, solution of polynomial equations and polynomial optimization. Tropical polynomials are encountered in a variety of applications, from computational geometry of polyhedra~\cite{Hampe2017Tropical} to image processing~\cite{Li1992Morphological}, cryptography~\cite{Grigoriev2014Tropical,Grigoriev2019Tropical} and games~\cite{Esparza2008Approximative,Gaubert2012Tropical}.

In the framework of tropical (algebraic) geometry, tropical polynomials arise as both valuable instruments and important objects of analysis. These polynomials are frequently defined over the min-plus real semifield with integer exponents (tropical Laurent polynomials), rational exponents (tropical Puiseux polynomials) and real exponents (generalized tropical Puiseux polynomials)~\cite{Itenberg2007Tropical,Markwig2010Field,Maclagan2015Introduction,Hampe2017Tropical,Grigoriev2018Tropical}. 

For particular semifields, the tropical polynomials are known to be convex functions in the ordinary sense, such as the polynomials over max-plus semifield, which are piecewise-linear convex functions. The optimization problems with the tropical polynomials as objective functions in these semirings can be solved using existing computational algorithms of convex optimization, including the simplex and Karmarkar algorithms in linear programming, and the subgradient and interior-point algorithms in nonlinear convex programming (see, e.g.,~\cite{Schrijver1998Theory,Borwein2000Convex} for both overviews of the approaches and detailed discussions). These algorithmic techniques often give indirect solutions in the form of approximate numerical values, but cannot provide precise analytical solutions that explicitly describe all possible solutions to the issue.

The purpose of this article is to discuss tropical polynomials in the general context of an arbitrary linearly ordered idempotent semifield with well-defined rational exponents. We consider optimization problems to minimize a tropical polynomial function of many variables, subject to box constraints on the variables. These issues arise in particular when addressing minimax approximation problems, such as single-facility location problems using Chebyshev and rectilinear distances~\cite{Krivulin2016Using,Krivulin2017Using,Krivulin2020Algebraic}, decision-making problems involving rating alternatives based on pairwise comparisons~\cite{Krivulin2019Tropical,Krivulin2020Using}, and others.

We propose a technique for variable elimination that converts the original optimization problem to a new one in which one variable is eliminated, and the box constraint for this variable is adjusted. We use this technique to develop a procedure to solve the problem in a finite number of iterations. The procedure includes two stages: backward elimination and forward substitution of variables. We describe the main steps of the procedure and discuss its computational complexity.

The proposed technique resembles the variable elimination method introduced by J.~Fourier in the early XIX century in~\cite{Fourier1825Analyse} to handle systems of linear inequalities. The method has been subsequently developed by L.~Dines~\cite{Dines1919Systems} and T.~Motzkin~\cite{Motzkin1936Contribution}, and is now known as the Fourier--Motzkin elimination~\cite{Ziegler1995Lectures,Schrijver1998Theory,Khachiyan2009Fourier} for linear problems in ordered fields. In order to overcome the recognized issue of the rapid growth of computational complexity of the Fourier--Motzkin elimination, a double description method has later been proposed in~\cite{Motzkin1953Double}, which reduces the complexity of the elimination. Both the Fourier--Motzkin elimination and the double description method appeared in the context of tropical mathematics in~\cite{Allamigeon2010Tropical,Allamigeon2014Tropical}, where tropical polyhedra are investigated in the setting of the max-plus algebra.

For polynomial optimization problems defined in terms of the max-plus semifield, the proposed two-stage procedure directly leads to a variant of the Fourier--Motzkin elimination for solving a linear program. In the case of arbitrary linearly ordered tropical semifields, the new technique can be considered as an extension of the Fourier--Motzkin method to solve generalized geometrical programs in the setting of linearly ordered tropical semifields. In its initial form, the procedure has a double-exponential complexity with respect to the number of monomials in the objective function, and can be reduced to exponential complexity by detecting and removing redundancy, which arises in the problem representation during elimination steps. Since the procedure becomes very time consuming as the numbers of variables and monomials increase, it may be computationally impractical to solve problems of high dimension. However, the proposed technique can be of value not only as a theoretical tool to study new classes of tropical optimization problems, but as a practical approach for modest-sized problems, that allows obtaining both direct exact solutions using rational arithmetic and approximate numerical solutions with floating-point computations.

The research described in this article was motivated by~\cite{Krivulin2016Using,Krivulin2017Using,Krivulin2020Algebraic} and their solution of tropical optimization problems. A general scheme for the technique is proposed in~\cite{Krivulin2020Usingparameter}, which develops a variable elimination method for unconstrained discrete linear Chebyshev approximation problems in terms of traditional mathematics. In the current research, we extend the elimination method to solve tropical polynomial optimization problems in the general setting of an arbitrary linearly ordered idempotent semifield. In this case, the discrete Chebyshev approximation problems can be considered as a special instance of the tropical optimization problem. Moreover, in contrast to the solution technique in~\cite{Krivulin2020Usingparameter}, which does not take into account constraints imposed on the unknown variables, the new technique allows solving problems with box constraints on the variables. 

The paper is organized as follows. In Section~\ref{S-TATP}, we present an overview of the basic definitions and notation, including the notion of tropical Puiseux polynomials, which underlie the results obtained in the next sections. In Section~\ref{S-POP}, we formulate the multivariate polynomial optimization problem of interest, make some comments on the problem and offer a complete solution of the problem in one unknown variable. Section~\ref{S-EV} includes an elimination lemma that provides the basis for the solution procedure of the multivariate optimization problem. In Section~\ref{S-SP}, we outline the proposed solution procedure and discuss its computational complexity. Section~\ref{S-AE} presents application examples and related numerical results. We draw some conclusions in Section~\ref{S-C}.

\section{Tropical Algebra and Tropical Polynomials}
\label{S-TATP}

This section introduces the fundamental definitions and notations that will be used in the subsequent sections to formulate and solve tropical polynomial optimization problems. For more information on tropical mathematics and its applications, consult~\cite{Baccelli1993Synchronization,Kolokoltsov1997Idempotent,Golan2003Semirings,Gondran2008Graphs,Itenberg2007Tropical,Maclagan2015Introduction}.

\subsection{Idempotent Semifield}

Let $\mathbb{X}$ be a nonempty set, which is equipped with addition $\oplus$ and multiplication $\otimes$, and has distinct elements zero $\mathbb{0}$ and one $\mathbb{1}$ such that $(\mathbb{X},\mathbb{0},\oplus)$ is an idempotent commutative monoid, $(\mathbb{X}\setminus\{\mathbb{0}\},\mathbb{1},\otimes)$ is an Abelian group, and multiplication distributes over addition. Under these conditions, the algebraic system $(\mathbb{X},\mathbb{0},\mathbb{1},\oplus,\otimes)$ is usually referred to as an idempotent (tropical) semifield. 

Idempotent addition satisfies the property $x\oplus x=x$ for all $x\in\mathbb{X}$ and induces a partial order on $\mathbb{X}$ by the rule: the relation $x\leq y$ holds for $x,y\in\mathbb{X}$ if and only if $x\oplus y=y$. From this definition follows that each $x$ satisfies the inequality $x\geq\mathbb{0}$. We assume that this partial order is even a linear order on $\mathbb{X}$. 

For each nonzero $x\in\mathbb{X}$, there exists a multiplicative inverse $x^{-1}$ such that $xx^{-1}=\mathbb{1}$ (here and henceforth, the multiplication sign $\otimes$ is omitted for compactness). The power notation with integer exponents indicates iterated multiplication: $\mathbb{0}^{n}=\mathbb{0}$, $x^{0}=\mathbb{1}$, $x^{n}=x^{n-1}x$ and $x^{-n}=(x^{-1})^{n}$ for all nonzero $x$ and positive integer $n$. We assume that the equation $x^{n}=a$ has a unique solution $x$ for each $a\in\mathbb{X}$ and integer $n>0$, and thus the semifield is radicable, which allows rational exponents. Moreover, the rational powers are assumed extended (by some appropriate limiting process) to real exponents. In what follows, the power notation is understood in terms of tropical algebra.  

With respect to the order induced by idempotent addition, both addition and multiplication are monotone: the inequality $x\leq y$ yields $x\oplus z\leq y\oplus z$ and $xz\leq yz$. Furthermore, addition possesses the extremal property (the majority law) in the form of the inequalities $x\leq x\oplus y$ and $y\leq x\oplus y$. The inequality $x\oplus y\leq z$ is equivalent to the pair of inequalities $x\leq z$ and $y\leq z$. Finally, exponentiation is monotone in the sense that, for any $x,y\ne\mathbb{0}$, the inequality $x\leq y$ results in $x^{r}\geq y^{r}$ if $r<0$, and $x^{r}\leq y^{r}$ if $r\geq0$.

Examples of the tropical semifield under consideration include the following algebraic systems (which are isomorphic to each other with obvious isomorphisms):
\begin{align*}
\mathbb{R}_{\max,+}
&=
(\mathbb{R}\cup\{-\infty\},-\infty,0,\max,+),
\\
\mathbb{R}_{\min,+}
&=
(\mathbb{R}\cup\{+\infty\},+\infty,0,\min,+),
\\
\mathbb{R}_{\max}
&=
(\mathbb{R}_{+}\cup\{0\},0,1,\max,\times),
\\
\mathbb{R}_{\min}
&=
(\mathbb{R}_{+}\cup\{+\infty\},+\infty,1,\min,\times),
\end{align*}
where $\mathbb{R}$ is the set of reals, and $\mathbb{R}_{+}=\{x>0|\ x\in\mathbb{R}\}$.

The semifield $\mathbb{R}_{\max,+}$ (also known as the max-plus algebra) has the operations $\oplus$ defined as taking maximum $\max$ and $\otimes$ as arithmetic addition $+$, with their neutral elements $\mathbb{0}$ given by $-\infty$ and $\mathbb{1}$ by $0$. For each $x\in\mathbb{R}$, the inverse $x^{-1}$ corresponds to the opposite number $-x$ in the conventional algebra; the power $x^{y}$ coincides with the usual product $xy$, and thus is defined for all $x,y\in\mathbb{R}$. The order induced by idempotent addition is the natural linear order on $\mathbb{R}$.

In the semifield $\mathbb{R}_{\min}$ (the min-algebra), the operations are $\oplus=\min$ and $\otimes=\times$, and the neutral elements are $\mathbb{0}=+\infty$ and $\mathbb{1}=1$. The inversion and exponentiation have the standard interpretation, whereas the order produced by the addition $\oplus$ is opposite to the linear order on $\mathbb{R}$.

\subsection{Tropical Puiseux Polynomials}

A tropical Puiseux monomial in $N$ variables $x_{1},\ldots,x_{N}$ over $\mathbb{X}$ is a product of powers $x_{1}^{p_{1}}\cdots x_{N}^{p_{N}}$ with exponents $p_{1},\ldots,p_{N}\in\mathbb{Q}$, where $\mathbb{Q}$ is the set of rationals. A tropical Puiseux polynomial is generally defined as a tropical linear combination of monomials $x_{1}^{p_{1}}\cdots x_{N}^{p_{N}}$ with given nonzero coefficients $a_{p_{1},\ldots,p_{N}}\in\mathbb{X}$ in the form
\begin{equation*}
f(x_{1},\ldots,x_{N})
=
\bigoplus_{(p_{1},\ldots,p_{N})\in P}
a_{p_{1},\ldots,p_{N}}
x_{1}^{p_{1}}\cdots x_{N}^{p_{N}},
\qquad
x_{j}\ne\mathbb{0},
\qquad
j=1,\ldots,N;
\end{equation*}
where $(p_{1},\ldots,p_{N})$ is the vector of exponents and $P\subset\mathbb{Q}^{N}$ is a finite subset.

Since all coefficients satisfy the condition $a_{p_{1},\ldots,p_{N}}>\mathbb{0}$, and hence are tropically positive, the function $f$ can also be considered as a tropical posynomial.

To simplify further formulas, we exploit an equivalent representation using a single index to label coefficients. Suppose that the polynomial $f$ consists of $M$ monomials, each given by its vector of exponents $(p_{1},\ldots,p_{N})$. Next, we assume that the monomials are ordered according to their vectors by using an appropriate ordering scheme (e.g., the lexicographic rule), and then consecutively numbered starting from 1 to  $M$. We use the numbers $i=1,\ldots,M$ to relabel the coefficients as $a_{i}$ and the components of the vectors of exponents as $(p_{i1},\ldots,p_{iN})$. As a result, the tropical Puiseux polynomial can be written as
\begin{equation*}
f(x_{1},\ldots,x_{N})
=
\bigoplus_{i=1}^{M}
a_{i}
x_{1}^{p_{i1}}\cdots x_{N}^{p_{iN}},
\qquad
x_{j}\ne\mathbb{0},
\qquad
j=1,\ldots,N;
\end{equation*}
where $a_{i}\in\mathbb{X}$, $a_{i}\ne\mathbb{0}$, and $p_{i1},\ldots,p_{iN}\in\mathbb{Q}$ for all $i=1,\ldots,M$.

Finally, we note that, in the context of the max-plus semifield $\mathbb{R}_{\max,+}$, the polynomial is represented in terms of the usual operations as
\begin{equation*}
f(x_{1},\ldots,x_{N})
=
\max_{1\leq i\leq M}
(p_{i1}x_{1}+\cdots+p_{iN}x_{N}+a_{i})
\end{equation*} 
and thus defines a piecewise-linear convex function on $\mathbb{R}^{N}$. In the framework of $\mathbb{R}_{\min}$, the conventional form of the polynomial becomes
\begin{equation*}
f(x_{1},\ldots,x_{N})
=
\min_{1\leq i\leq M}
a_{i}
x_{1}^{p_{i1}}\cdots x_{N}^{p_{iN}},
\end{equation*} 
which specifies a nonlinear concave (upper convex) function on $\mathbb{R}_{+}^{N}$.

\section{Polynomial Optimization Problems}
\label{S-POP}

We are concerned with constrained optimization problems in the tropical algebra setting, where one needs to minimize, in terms of the order induced by idempotent addition, an objective function given by a tropical Puiseux polynomial, subject to box constraints on the unknown variables. Given parameters $p_{ij}\in\mathbb{Q}$ and $a_{i},g_{j},h_{j}\in\mathbb{X}$ such that $a_{i},h_{j}\ne\mathbb{0}$ and $g_{j}\leq h_{j}$ for all $i=1,\ldots,M$ and $j=1,\ldots,N$, the problem is to find nonzero $x_{1},\ldots,x_{N}\in\mathbb{X}$ that give
\begin{equation}
\begin{aligned}
&&
\min_{x_{1},\ldots,x_{N}}
&&&
\bigoplus_{i=1}^{M}
a_{i}x_{1}^{p_{i1}}\cdots x_{N}^{p_{iN}};
\\
&&
\text{s.t.}
&&&
g_{j}
\leq
x_{j}
\leq
h_{j},
\qquad
j=1,\ldots,N.
\end{aligned}
\label{P-minx1xNaix1pi1xNpiN-gjleqxjleqhj}
\end{equation}

Note that this problem can also be considered as a tropical analogue of a constrained geometric program.

The operator $\min$ in the formulation of problem~\eqref{P-minx1xNaix1pi1xNpiN-gjleqxjleqhj} is understood in the framework of the order on $\mathbb{X}$, and thus the conventional interpretation of the optimization objective is dependent on the particular semifield. Specifically, if the problem is given in terms of the max-plus algebra $\mathbb{R}_{\max,+}$, it is a minimization problem in the ordinary sense as well.

When considered in the framework of min-algebra $\mathbb{R}_{\min}$, this problem corresponds to an ordinary maximization problem since the objective $\min$ is treated in the sense of an order that is opposite to the natural linear order.

We observe that the polynomial optimization problems formulated in the sense of the semifields $\mathbb{R}_{\max,+}$ and $\mathbb{R}_{\min,+}$ can be represented as linear programs and then solved by appropriate computational schemes of linear programming, including the simplex and Karmarkar algorithms.

In terms of $\mathbb{R}_{\max}$ or $\mathbb{R}_{\min}$, problem~\eqref{P-minx1xNaix1pi1xNpiN-gjleqxjleqhj} becomes a nonlinear convex or concave optimization problem, which can be handled using solution techniques available in convex programming, such as subgradient and interior-point algorithms.

These algorithmic approaches, which are focused on numerical solutions, cannot guarantee a direct solution in an explicit form. Below, we apply a tropical algebraic technique to solve the problem with one unknown variable in a general setting of an arbitrary idempotent semifield. In the next sections, we extend this solution to prove an elimination lemma that provides a basis for a complete analytical solution of the multivariate polynomial optimization problems.

In the case of one-variable polynomials, an analytical solution to the problem can be obtained as follows. Consider problem~\eqref{P-minx1xNaix1pi1xNpiN-gjleqxjleqhj} with $N=1$ and represent it in a simplified form without indices that indicate the variable number, as the problem
\begin{equation}
\begin{aligned}
&&
\min_{x}
&&&
\bigoplus_{i=1}^{M}
a_{i}x^{p_{i}};
\\
&&
\text{s.t.}
&&&
g
\leq
x
\leq
h.
\end{aligned}
\label{P-minx1aix1pi1-g1leqx1leqh}
\end{equation}

The next result offers a complete direct solution to the problem.
\begin{proposition}
\label{N-minx1aix1pi1-g1leqx1leqh}
The minimum value of the objective function in problem~\eqref{P-minx1aix1pi1-g1leqx1leqh} is equal to
\begin{equation}
\mu
=
\bigoplus_{\substack{1\leq i,k\leq M\\p_{i}<0,\ p_{k}>0}}
a_{i}^{-\frac{p_{k}}{p_{i}-p_{k}}}a_{k}^{\frac{p_{i}}{p_{i}-p_{k}}}
\oplus
\bigoplus_{1\leq i\leq M}
(h^{-p_{i}}
\oplus
g^{-p_{i}})^{-1}
a_{i},
\label{E-mu1}
\end{equation}
whereas all solutions are given by the condition
\begin{equation}
\bigoplus_{\substack{1\leq i\leq M\\p_{i}<0}}
\mu^{1/p_{i}}
a_{i}^{-1/p_{i}}
\oplus
g
\leq
x
\leq
\left(
\bigoplus_{\substack{1\leq i\leq M\\p_{i}>0}}
\mu^{-1/p_{i}}
a_{i}^{1/p_{i}}
\oplus
h^{-1}
\right)^{-1},
\label{I-x1}
\end{equation}
where and thereafter the empty sums are interpreted as $\mathbb{0}$.
\end{proposition}

\begin{proof}
First, we use an auxiliary variable $\lambda$ to rewrite the problem as the following extended constrained problem with two variables:
\begin{equation*}
\begin{aligned}
&&
\min_{x,\lambda}
&&&
\lambda;
\\
&&
\text{s.t.}
&&&
\bigoplus_{i=1}^{M}
a_{i}x^{p_{i}}
\leq
\lambda,
\\
&&&&&
g
\leq
x
\leq
h.
\end{aligned}
\end{equation*}

We consider the first inequality constraint and apply properties of idempotent addition to represent it in equivalent form as the system of inequalities
\begin{equation*}
\lambda
\geq
a_{i}x^{p_{i}};
\qquad
i=1,\ldots,M.
\end{equation*}

Since both $a_{i}\ne\mathbb{0}$ for all $i$ and $x\ne\mathbb{0}$ by assumption, the variable $\lambda$ is bounded from below as $\lambda\geq a_{i}x^{p_{i}}>\mathbb{0}$.

After solving for $x$ those inequalities which have nonzero exponents of $x$, the system takes the form
\begin{align*}
x
&\geq
\lambda^{1/p_{i}}
a_{i}^{-1/p_{i}},
&&
p_{i}<0;
\\
x^{-1}
&\geq
\lambda^{-1/p_{i}}
a_{i}^{1/p_{i}},
&&
p_{i}>0;
\\
\lambda
&\geq
a_{i},
&&
p_{i}=0;
\qquad
i=1,\ldots,M.
\end{align*}

Combining the inequalities with common left-hand sides and adding the box constraint $g\leq x\leq h$ rewritten as two inequalities $x\geq g$ and $x^{-1}\geq h^{-1}$ yield
\begin{equation}
\begin{aligned}
x
&\geq
\bigoplus_{p_{i}<0}
\lambda^{1/p_{i}}
a_{i}^{-1/p_{i}}
\oplus
g,
\\
x^{-1}
&\geq
\bigoplus_{p_{i}>0}
\lambda^{-1/p_{i}}
a_{i}^{1/p_{i}}
\oplus
h^{-1},
\\
\lambda
&\geq
\bigoplus_{p_{i}=0}
a_{i}.
\end{aligned}
\label{I-x-x1-lambda}
\end{equation}

We take the first two inequalities at~\eqref{I-x-x1-lambda} to couple into the double inequality
\begin{equation}
\bigoplus_{p_{i}<0}
\lambda^{1/p_{i}}
a_{i}^{-1/p_{i}}
\oplus
g
\leq
x
\leq
\left(
\bigoplus_{p_{i}>0}
\lambda^{-1/p_{i}}
a_{i}^{1/p_{i}}
\oplus
h^{-1}
\right)^{-1}.
\label{I-leqxleq}
\end{equation}

This inequality specifies a consistent boundary condition for the unknown $x$ if and only if the inequality
\begin{equation*}
\bigoplus_{p_{i}<0}
\lambda^{1/p_{i}}
a_{i}^{-1/p_{i}}
\oplus
g
\leq
\left(
\bigoplus_{p_{i}>0}
\lambda^{-1/p_{i}}
a_{i}^{1/p_{i}}
\oplus
h^{-1}
\right)^{-1}
\end{equation*}
holds, which is equivalent to the inequality
\begin{equation*}
\left(
\bigoplus_{p_{i}<0}
\lambda^{1/p_{i}}
a_{i}^{-1/p_{i}}
\oplus
g
\right)
\left(
\bigoplus_{p_{k}>0}
\lambda^{-1/p_{k}}
a_{k}^{1/p_{k}}
\oplus
h^{-1}
\right)
\leq
\mathbb{1}.
\end{equation*}

We now solve the obtained inequality for $\lambda$. After expanding the left-hand side, we replace this inequality by the equivalent system of four inequalities
\begin{align*}
\bigoplus_{p_{i}<0,\ p_{k}>0}
\lambda^{1/p_{i}-1/p_{k}}
a_{i}^{-1/p_{i}}a_{k}^{1/p_{k}}
&\leq
\mathbb{1},
\\
h^{-1}
\bigoplus_{p_{i}<0}
\lambda^{1/p_{i}}
a_{i}^{-1/p_{i}}
&\leq
\mathbb{1},
\\
g
\bigoplus_{p_{k}>0}
\lambda^{-1/p_{k}}
a_{k}^{1/p_{k}}
&\leq
\mathbb{1},
\\
gh^{-1}
&\leq
\mathbb{1}.
\end{align*}

We note that the last inequality is equivalent to $g\leq h$, and hence is valid.

Consider the first inequality of this system, written in the form
\begin{equation*}
\bigoplus_{p_{i}<0,\ p_{k}>0}
\lambda^{-\frac{p_{i}-p_{k}}{p_{i}p_{k}}}
a_{i}^{-\frac{1}{p_{i}}}a_{k}^{\frac{1}{p_{k}}}
\leq
\mathbb{1},
\end{equation*}
and observe that the exponents of $\lambda$ in all terms satisfy the condition
\begin{equation*}
-\frac{p_{i}-p_{k}}{p_{i}p_{k}}
<0.
\end{equation*}

In a similar way as before, we represent the first inequality as a set of inequalities, one for each $i$ and $k$. We solve all inequalities in this set for $\lambda$ and then combine the results into one inequality to bound $\lambda$ from below as
\begin{equation*}
\lambda
\geq
\bigoplus_{p_{i}<0,\ p_{k}>0}
a_{i}^{-\frac{p_{k}}{p_{i}-p_{k}}}a_{k}^{\frac{p_{i}}{p_{i}-p_{k}}}.
\end{equation*}

Furthermore, application of the same solution technique to the second and third inequality of the above system yields the result 
\begin{align*}
\lambda
&\geq
\bigoplus_{p_{i}<0}
h^{p_{i}}
a_{i},
\\
\lambda
&\geq
\bigoplus_{p_{i}>0}
g^{p_{i}}
a_{i}.
\end{align*}

Let us verify that these two inequalities together with the last inequality at~\eqref{I-x-x1-lambda} are equivalent to an inequality that offers another lower bound for $\lambda$ as follows:
\begin{equation*}
\lambda
\geq
\bigoplus_{1\leq i\leq M}
(h^{-p_{i}}
\oplus
g^{-p_{i}})^{-1}
a_{i}.
\end{equation*}

Indeed, with the condition $g\leq h$, for each $i$, we obtain
\begin{equation*}
(h^{-p_{i}}
\oplus
g^{-p_{i}})^{-1}
=
\begin{cases}
h^{p_{i}},
&
\text{if $p_{i}<0$};
\\
\mathbb{1},
&
\text{if $p_{i}=0$};
\\
g^{p_{i}},
&
\text{if $p_{i}>0$}.
\end{cases}
\end{equation*}

Substitution into the above inequality allows us to split its right-hand side into three sums and then rewrite the inequality as three inequalities corresponding to the positive, zero and negative exponents $p_{i}$.

Finally, we combine both lower bounds into one bound
\begin{equation*}
\lambda
\geq
\bigoplus_{\substack{1\leq i,k\leq M\\p_{i}<0,\ p_{k}>0}}
a_{i}^{-\frac{p_{k}}{p_{i}-p_{k}}}a_{k}^{\frac{p_{i}}{p_{i}-p_{k}}}
\oplus
\bigoplus_{1\leq i\leq M}
(h^{-p_{i}}
\oplus
g^{-p_{i}})^{-1}
a_{i}.
\end{equation*}

We take the right-hand side of this inequality as the minimum of $\lambda$ in the extended problem, which is the minimum in problem~\eqref{P-minx1aix1pi1-g1leqx1leqh} as well. We denote this minimum by $\mu$ to write equality~\eqref{E-mu1}.

All solutions $x$ that achieve this minimum are given by inequality~\eqref{I-leqxleq} with $\lambda$ substituted by $\mu$, which yields~\eqref{I-x1}.
\end{proof}

Note that the number of tropical summands in the first sum at~\eqref{E-mu1} attains its maximum $\lfloor M^{2}/4\rfloor$, when there are no zero exponents $p_{i}$ and the number of positive and negative exponents are minimally different.

\section{Elimination of Variables}
\label{S-EV}

In this section, we demonstrate how the solution of the polynomial optimization problem with $N>1$ variables can be reduced to the solution of a problem of the same form, but with one unknown variable less. This result serves as the key component of a solution procedure presented below, which allows one to solve the problem in a finite number of iterations. 

We describe an algebraic transformation technique used to eliminate a variable from the objective function, and then to rearrange this function and modify the box constraint for the variable eliminated. The technique follows similar principles to the Fourier--Motzkin elimination method~\cite{Ziegler1995Lectures,Schrijver1998Theory,Khachiyan2009Fourier}, and extends this method, initially designed in the framework of ordered fields, to optimization problems in tropical semifields. 

In the same way as in Proposition~\ref{N-minx1aix1pi1-g1leqx1leqh}, we introduce an auxiliary parameter to represent the minimum value of the objective function and to replace this function with a system of parameterized inequalities. The system is solved with respect to the variable $x_{N}$, and the solution obtained is combined with the box constraint for $x_{N}$ to derive a new parameterized box constraint. The condition for the new constraint to be consistent is used to establish a lower bound for the parameter, which serves as an objective function in the reduced problem. 

The formal basis of the technique is provided by the following result.
\begin{lemma}
\label{L-minx1xNaix1pi1xNpiN-gjleqxjleqhj}
Problem~\eqref{P-minx1xNaix1pi1xNpiN-gjleqxjleqhj} is equivalent to the problem
\begin{equation}
\begin{aligned}
&&
\min_{x_{1},\ldots,x_{N-1}}
&&&
\bigoplus_{i=1}^{M^{2}+M}
b_{i}x_{1}^{q_{i1}}\cdots x_{N-1}^{q_{i,N-1}};
\\
&&
\text{\upshape s.t.}
&&&
g_{j}
\leq
x_{j}
\leq
h_{j},
\qquad
j=1,\ldots,N-1;
\end{aligned}
\label{P-minx1xN1bix1qi1xN1qiN1}
\end{equation}
together with the double inequality
\begin{equation}
\bigoplus_{i=1}^{M}
c_{i}x_{1}^{r_{i1}}\cdots x_{N-1}^{r_{i,N-1}}
\oplus
g_{N}
\leq
x_{N}
\leq
\left(
\bigoplus_{i=1}^{M}
d_{i}x_{1}^{s_{i1}}\cdots x_{N-1}^{s_{i,N-1}}
\oplus
h_{N}^{-1}
\right)^{-1},
\label{I-cixirleqxNleqdixis}
\end{equation}
where and thereafter the empty products are interpreted as $\mathbb{1}$, and for all $i,k=1,\ldots,M$ and $j=1,\ldots,N-1$, the following notation is used:
\begin{equation}
\begin{aligned}
q_{M(i-1)+k,j}
&=
\begin{cases}
-
\frac{p_{ij}p_{kN}-p_{kj}p_{iN}}{p_{iN}-p_{kN}},
&
\text{if $p_{iN}<0$ and $p_{kN}>0$};
\\
0,
&
\text{otherwise};
\end{cases}
\\
q_{M^{2}+i,j}
&=
p_{ij};
\\
r_{ij}
&=
\begin{cases}
-
p_{ij}/p_{iN},
&
\text{if $p_{iN}<0$};
\\
0,
&
\text{otherwise};
\end{cases}
\\
s_{ij}
&=
\begin{cases}
p_{ij}/p_{iN},
&
\text{if $p_{iN}>0$};
\\
0,
&
\text{otherwise};
\end{cases}
\\
b_{M(i-1)+k}
&=
\begin{cases}
a_{i}^{-\frac{p_{kN}}{p_{iN}-p_{kN}}}a_{k}^{\frac{p_{iN}}{p_{iN}-p_{kN}}},
&
\text{if $p_{iN}<0$ and $p_{kN}>0$};
\\
\mathbb{0},
&
\text{otherwise};
\end{cases}
\\
b_{M^{2}+i}
&=
(h_{N}^{-p_{iN}}
\oplus
g_{N}^{-p_{iN}})^{-1}
a_{i};
\\
c_{i}
&=
\begin{cases}
\mu^{1/p_{iN}}
a_{i}^{-1/p_{iN}},
&
\text{if $p_{iN}<0$};
\\
\mathbb{0},
&
\text{otherwise};
\end{cases}
\\
d_{i}
&=
\begin{cases}
\mu^{-1/p_{iN}}
a_{i}^{1/p_{iN}},
&
\text{if $p_{iN}>0$};
\\
\mathbb{0},
&
\text{otherwise};
\end{cases}
\end{aligned}
\label{E-qij-rij-sij-bik-ci-di}
\end{equation}
and $\mu$ is the minimum of the objective function in problem~\eqref{P-minx1xN1bix1qi1xN1qiN1}.
\end{lemma}

\begin{proof}
With an auxiliary parameter $\lambda$, problem~\eqref{P-minx1xNaix1pi1xNpiN-gjleqxjleqhj} is represented as
\begin{equation*}
\begin{aligned}
&&
\min_{x_{1},\ldots,x_{N},\lambda}
&&&
\lambda;
\\
&&
\text{s.t.}
&&&
\bigoplus_{i=1}^{M}
a_{i}x_{1}^{p_{i1}}\cdots x_{N}^{p_{iN}}
\leq
\lambda,
\\
&&&&&
g_{j}
\leq
x_{j}
\leq
h_{j},
\qquad
j=1,\ldots,N.
\end{aligned}
\end{equation*}

We replace the first inequality constraint by the system of inequalities
\begin{equation*}
\lambda
\geq
a_{i}x_{1}^{p_{i1}}\cdots x_{N}^{p_{iN}};
\qquad
i=1,\ldots,M;
\end{equation*}
and then solve each inequality for $x_{N}$ to obtain the inequalities
\begin{align*}
x_{N}
&\geq
\lambda^{1/p_{iN}}
a_{i}^{-1/p_{iN}}x_{1}^{-p_{i1}/p_{iN}}\cdots x_{N-1}^{-p_{i,N-1}/p_{iN}},
&&
p_{iN}<0;
\\
x_{N}^{-1}
&\geq
\lambda^{-1/p_{iN}}
a_{i}^{1/p_{iN}}x_{1}^{p_{i1}/p_{iN}}\cdots x_{N-1}^{p_{i,N-1}/p_{iN}},
&&
p_{iN}>0;
\\
\lambda
&\geq
a_{i}x_{1}^{p_{i1}}\cdots x_{N-1}^{p_{i,N-1}},
&&
p_{iN}=0;
\qquad
i=1,\ldots,M.
\end{align*}

Combining the inequalities obtained and the box constraint for $x_{N}$ results in the~system
\begin{equation}
\begin{aligned}
x_{N}
&\geq
\bigoplus_{\substack{1\leq i\leq M\\p_{iN}<0}}
\lambda^{1/p_{iN}}
a_{i}^{-1/p_{iN}}x_{1}^{-p_{i1}/p_{iN}}\cdots x_{N-1}^{-p_{i,N-1}/p_{iN}}
\oplus
g_{N},
\\
x_{N}^{-1}
&\geq
\bigoplus_{\substack{1\leq i\leq M\\p_{iN}>0}}
\lambda^{-1/p_{iN}}
a_{i}^{1/p_{iN}}x_{1}^{p_{i1}/p_{iN}}\cdots x_{N-1}^{p_{i,N-1}/p_{iN}}
\oplus
h_{N}^{-1},
\\
\lambda
&\geq
\bigoplus_{\substack{1\leq i\leq M\\p_{iN}=0}}
a_{i}x_{1}^{p_{i1}}\cdots x_{N-1}^{p_{i,N-1}},
\end{aligned}.
\label{I-xN-xN1-lambda}
\end{equation}

The first two inequalities in~\eqref{I-xN-xN1-lambda} lead to the double inequality
\begin{multline}
\bigoplus_{\substack{1\leq i\leq M\\p_{iN}<0}}
\lambda^{1/p_{iN}}
a_{i}^{-1/p_{iN}}x_{1}^{-p_{i1}/p_{iN}}\cdots x_{N-1}^{-p_{i,N-1}/p_{iN}}
\oplus
g_{N}
\leq
x_{N}
\\\leq
\left(
\bigoplus_{\substack{1\leq i\leq M\\p_{iN}>0}}
\lambda^{-1/p_{iN}}
a_{i}^{1/p_{iN}}x_{1}^{p_{i1}/p_{iN}}\cdots x_{N-1}^{p_{i,N-1}/p_{iN}}
\oplus
h_{N}^{-1}
\right)^{-1}.
\label{I-leqxNleq}
\end{multline}

This double inequality implies that its left part is less or equal to the right, which is equivalent to the inequality 
\begin{multline*}
\left(
\bigoplus_{\substack{1\leq i\leq M\\p_{iN}<0}}
\lambda^{1/p_{iN}}
a_{i}^{-1/p_{iN}}x_{1}^{-p_{i1}/p_{iN}}\cdots x_{N-1}^{-p_{i,N-1}/p_{iN}}
\oplus
g_{N}
\right)
\\\otimes
\left(
\bigoplus_{\substack{1\leq k\leq M\\p_{kN}>0}}
\lambda^{-1/p_{kN}}
a_{k}^{1/p_{kN}}x_{1}^{p_{k1}/p_{kN}}\cdots x_{N-1}^{p_{k,N-1}/p_{kN}}
\oplus
h_{N}^{-1}
\right)
\leq
\mathbb{1}.
\end{multline*}

By expanding double brackets, we split the inequality into four inequalities 
\begin{multline*}
\bigoplus_{\substack{1\leq i,k\leq M\\p_{iN}<0,\ p_{kN}>0}}
\lambda^{1/p_{iN}-1/p_{kN}}
a_{i}^{-1/p_{iN}}a_{k}^{1/p_{kN}}
\\
\begin{aligned}
\otimes
x_{1}^{-p_{i1}/p_{iN}+p_{k1}/p_{kN}}\cdots x_{N-1}^{-p_{i,N-1}/p_{iN}+p_{k,N-1}/p_{kN}}
&\leq
\mathbb{1},
\\
h_{N}^{-1}
\bigoplus_{\substack{1\leq i\leq M\\p_{iN}<0}}
\lambda^{1/p_{iN}}
a_{i}^{-1/p_{iN}}x_{1}^{-p_{i1}/p_{iN}}\cdots x_{N-1}^{-p_{i,N-1}/p_{iN}}
&\leq
\mathbb{1},
\\
g_{N}
\bigoplus_{\substack{1\leq k\leq M\\p_{kN}>0}}
\lambda^{-1/p_{kN}}
a_{k}^{1/p_{kN}}x_{1}^{p_{k1}/p_{kN}}\cdots x_{N-1}^{p_{k,N-1}/p_{kN}}
&\leq
\mathbb{1},
\\
g_{N}h_{N}^{-1}
&\leq
\mathbb{1},
\end{aligned}
\end{multline*}
where the last inequality trivially holds by assumption.

The solution of the first inequality for $\lambda$ yields a lower bound for $\lambda$, given by
\begin{equation*}
\lambda
\geq
\bigoplus_{\substack{1\leq i,k\leq M\\p_{iN}<0,\ p_{kN}>0}}
a_{i}^{-\frac{p_{kN}}{p_{iN}-p_{kN}}}a_{k}^{\frac{p_{iN}}{p_{iN}-p_{kN}}}
x_{1}^{-\frac{p_{i1}p_{kN}-p_{k1}p_{iN}}{p_{iN}-p_{kN}}}\cdots x_{N-1}^{-\frac{p_{i,N-1}p_{kN}-p_{k,N-1}p_{iN}}{p_{iN}-p_{kN}}}.
\end{equation*}

Next, we solve the second and third inequality for $\lambda$ to write the inequalities
\begin{align*}
\lambda
&\geq
\bigoplus_{\substack{1\leq i\leq M\\p_{iN}<0}}
h_{N}^{p_{iN}}
a_{i}x_{1}^{p_{i1}}\cdots x_{N-1}^{p_{i,N-1}},
\\
\lambda
&\geq
\bigoplus_{\substack{1\leq i\leq M\\p_{iN}>0}}
g_{N}^{p_{iN}}
a_{i}x_{1}^{p_{i1}}\cdots x_{N-1}^{p_{i,N-1}},
\end{align*}
and then combine them with the last inequality at~\eqref{I-xN-xN1-lambda} to derive another bound
\begin{equation*}
\lambda
\geq
\bigoplus_{1\leq i\leq M}
(h_{N}^{-p_{iN}}
\oplus
g_{N}^{-p_{iN}})^{-1}
a_{i}x_{1}^{p_{i1}}\cdots x_{N-1}^{p_{i,N-1}}.
\end{equation*}

By coupling both lower bounds, we obtain 
\begin{multline*}
\lambda
\geq
\bigoplus_{\substack{1\leq i,k\leq M\\p_{iN}<0,\ p_{kN}>0}}
a_{i}^{-\frac{p_{kN}}{p_{iN}-p_{kN}}}a_{k}^{\frac{p_{iN}}{p_{iN}-p_{kN}}}
x_{1}^{-\frac{p_{i1}p_{kN}-p_{k1}p_{iN}}{p_{iN}-p_{kN}}}\cdots x_{N-1}^{-\frac{p_{i,N-1}p_{kN}-p_{k,N-1}p_{iN}}{p_{iN}-p_{kN}}}
\\\oplus
\bigoplus_{1\leq i\leq M}
(h_{N}^{-p_{iN}}
\oplus
g_{N}^{-p_{iN}})^{-1}
a_{i}x_{1}^{p_{i1}}\cdots x_{N-1}^{p_{i,N-1}}.
\end{multline*}

With the notation given by~\eqref{E-qij-rij-sij-bik-ci-di}, the last inequality becomes
\begin{equation*}
\lambda
\geq
\bigoplus_{i=1}^{M^{2}+M}
b_{i}x_{1}^{q_{i1}}\cdots x_{N-1}^{q_{i,N-1}}.
\end{equation*}

It is easy to see that a sufficient part of the coefficients $b_{i}$ for $i=1,\ldots,M^{2}$ are equal to $\mathbb{0}$. Indeed, it follows from the definition of $b_{i}$ at~\eqref{E-qij-rij-sij-bik-ci-di} that the maximum number of nonzero coefficients among the first $M^{2}$ is bounded from above by $M^{2}/4$. As a result, the actual number of monomials in the polynomial on the right-hand side of the above inequality is not greater than $M^{2}/4+M$.

It remains to see that finding the minimum value of $\lambda$ in this inequality under the box-constraints $g_{j}\leq x_{j}\leq h_{j}$ for all $j=1,\ldots,N-1$ is equivalent to the solution of problem~\eqref{P-minx1xN1bix1qi1xN1qiN1}.

At the same time, with the minimum of $\lambda$ denoted by $\mu$ and the notation at~\eqref{E-qij-rij-sij-bik-ci-di}, the box constraint for $x_{N}$ in the form of~\eqref{I-leqxNleq} reduces to~\eqref{I-cixirleqxNleqdixis}.
\end{proof}

To conclude this section, we note that the proof of the lemma remains valid if we allow the tropical Puiseux polynomials to have real exponents, and thus extend the result to cover optimization problems with generalized tropical Puiseux polynomials.

\section{Solution Procedure}
\label{S-SP}

We now describe a solution procedure that applies the elimination lemma to find all solutions of problem~\eqref{P-minx1xNaix1pi1xNpiN-gjleqxjleqhj} in a finite number of steps. As in the process of the Fourier--Motzkin elimination, the procedure includes two phases: (i) backward elimination and (ii) forward substitution of variables.

\subsection{Backward Elimination and Forward Substitution}
The backward elimination phase consists of successive steps of eliminating the variables from $x_{N}$ to $x_{1}$. Upon completion of this phase, the objective function of the problem reduces to a constant that is equal to the minimum value of this function. As another result, both lower and upper bounds for each variable are derived in the form of polynomial functions of other variables and the minimum value of the objective function. 

The forward substitution phase involves steps of evaluating the lower and upper bounds for each variable from $x_{1}$ to $x_{N}$. The phase starts with calculating the bounds for the variable $x_{1}$, which are dependent only on the minimum of the objective function. Substitution of a value of $x_{1}$, which satisfies these bounds, yields the bounds for $x_{2}$, and so on. The last step produces lower and upper bounds for the variable $x_{N}$, and thereby completes the solution.

To describe the procedure more formally, first consider the variable elimination phase. For each $n=N,N-1,\ldots,1$, the procedure gradually rearranges problem~\eqref{P-minx1xNaix1pi1xNpiN-gjleqxjleqhj} into the~problems
\begin{equation*}
\begin{aligned}
&&
\min_{x_{1},\ldots,x_{n-1}}
&&&
\bigoplus_{i=1}^{M_{n-1}}
a_{i}^{(n-1)}x_{1}^{p_{i1}^{(n-1)}}\cdots x_{n-1}^{p_{i,n-1}^{(n-1)}};
\\
&&
\text{\upshape s.t.}
&&&
g_{j}
\leq
x_{j}
\leq
h_{j},
\qquad
j=1,\ldots,n-1;
\end{aligned}
\end{equation*}
where $M_{n-1}=M_{n}^{2}+M_{n}$ with $M_{N}=M$, and for all $i,k=1,\ldots,M_{n}$ and $j=1,\ldots,n-1$, the following recurrent relations are used:
\begin{align*}
p_{M_{n}(i-1)+k,j}^{(n-1)}
&=
\begin{cases}
-
\frac{p_{ij}^{(n)}p_{kn}^{(n)}-p_{kj}^{(n)}p_{in}^{(n)}}{p_{in}^{(n)}-p_{kn}^{(n)}},
&
\text{if $p_{in}^{(n)}<0$ and $p_{kn}^{(n)}>0$};
\\
0,
&
\text{otherwise};
\end{cases}
\\
p_{M_{n}^{2}+i,j}^{(n-1)}
&=
p_{ij}^{(n)};
\\
a_{M_{n}(i-1)+k}^{(n-1)}
&=
\begin{cases}
(a_{i}^{(n)})^{-\frac{p_{kn}^{(n)}}{p_{in}^{(n)}-p_{kn}^{(n)}}}(a_{k}^{(n)})^{\frac{p_{in}^{(n)}}{p_{in}^{(n)}-p_{kn}^{(n)}}},
&
\text{if $p_{in}^{(n)}<0$ and $p_{kn}^{(n)}>0$};
\\
\mathbb{0},
&
\text{otherwise};
\end{cases}
\\
a_{M_{n}^{2}+i}^{(n-1)}
&=
(h_{n}^{-p_{in}^{(n)}}
\oplus
g_{n}^{-p_{in}^{(n)}})^{-1}
a_{i}^{(n)}
\end{align*}
together with the conditions $a_{i}^{(N)}=a_{i}$ and $p_{ij}^{(N)}=p_{ij}$.

The phase ends at $n=1$, when the objective function becomes the constant
\begin{equation*}
\mu
=
\bigoplus_{i=1}^{M_{0}}
a_{i}^{(0)},
\end{equation*} 
which specifies the minimum value of the objective function in problem~\eqref{P-minx1xNaix1pi1xNpiN-gjleqxjleqhj}.

Furthermore, the procedure constructs a system of box constraints given for each $n=N,N-1,\ldots,1$ by the double inequalities
\begin{equation*}
\bigoplus_{i=1}^{M_{n}}
c_{i}^{(n-1)}x_{1}^{r_{i1}^{(n-1)}}\cdots x_{n-1}^{r_{i,n-1}^{(n-1)}}
\oplus
g_{n}
\leq
x_{n}
\leq
\left(
\bigoplus_{i=1}^{M_{n}}
d_{i}^{(n-1)}x_{1}^{s_{i1}^{(n-1)}}\cdots x_{n-1}^{s_{i,n-1}^{(n-1)}}
\oplus
h_{n}^{-1}
\right)^{-1},
\end{equation*}
using for all $i=1,\ldots,M_{n}$ and $j=1,\ldots,n-1$, the following recurrent relations:
\begin{align*}
r_{ij}^{(n-1)}
&=
\begin{cases}
-
p_{ij}^{(n)}/p_{in}^{(n)},
&
\text{if $p_{in}^{(n)}<0$};
\\
0,
&
\text{otherwise};
\end{cases}
\\
s_{ij}^{(n-1)}
&=
\begin{cases}
p_{ij}^{(n)}/p_{in}^{(n)},
&
\text{if $p_{in}^{(n)}>0$};
\\
0,
&
\text{otherwise};
\end{cases}
\\
c_{i}^{(n-1)}
&=
\begin{cases}
\mu^{1/p_{in}^{(n)}}
(a_{i}^{(n)})^{-1/p_{in}^{(n)}},
&
\text{if $p_{in}^{(n)}<0$};
\\
\mathbb{0},
&
\text{otherwise};
\end{cases}
\\
d_{i}^{(n-1)}
&=
\begin{cases}
\mu^{-1/p_{in}^{(n)}}
(a_{i}^{(n)})^{1/p_{in}^{(n)}},
&
\text{if $p_{in}^{(n)}>0$};
\\
\mathbb{0},
&
\text{otherwise}.
\end{cases}
\end{align*}

The forward substitution phase exploits the box constraints to calculate, one by one, the lower and upper bounds for the unknown variables $x_{n}$ for all $n=1,\ldots,N$. The phase begins with the derivation of the box constraint for $x_{1}$, which is given by constant bounds~as
\begin{equation*}
\bigoplus_{i=1}^{M_{1}}
c_{i}^{(0)}
\oplus
g_{1}
\leq
x_{1}
\leq
\left(
\bigoplus_{i=1}^{M_{1}}
d_{i}^{(0)}
\oplus
h_{1}^{-1}
\right)^{-1}.
\end{equation*}

As the next step, substitution of a value for $x_{1}$, which satisfies the inequality, into the box constraint for $x_{2}$ yields a box constraint with constant bounds for $x_{2}$. In the same way, further steps make it possible to fix the values of the other variables. 

The core of the procedure is recalculating the coefficients and exponents in the monomials, which compose the polynomial functions that represent the objectives and constraints during the elimination process. We combine the coefficients into the vectors 
\begin{equation*}
\bm{a}_{n}
=
(a_{i}^{(n)}),
\qquad
\bm{c}_{n}
=
(c_{i}^{(n)}),
\qquad
\bm{d}_{n}
=
(d_{i}^{(n)}),
\end{equation*}
and the exponents into the matrices
\begin{equation*}
\bm{P}_{n}
=
(p_{ij}^{(n)}),
\qquad
\bm{R}_{n}
=
(r_{ij}^{(n)}),
\qquad
\bm{S}_{n}
=
(s_{ij}^{(n)}).
\end{equation*}

We now summarize the transformation of these arrays in the form of Algorithm~\ref{A-Forward-Backward}.
\begin{algorithm}
\caption{Forward-Backward}
\label{A-Forward-Backward}
\setstretch{1.4}
\renewcommand{\COMMENT}[1]{\mbox{\bfseries comment:}\ \mbox{#1}}
\begin{pseudocode}{\hl{Forward-Backward}}{\bm{P},\bm{a},\bm{g},\bm{h},M,N}\label{Forward-Backward}
\PROCEDURE{Objectives}{\bm{P}_{n-1},\bm{a}_{n-1},\bm{P}_{n},\bm{a}_{n},M_{n},n}
\COMMENT{Construct $\bm{P}_{n-1}$ and $\bm{a}_{n-1}$ from $\bm{P}_{n}$ and $\bm{a}_{n}$}
\\
\FOR i \GETS 1 \TO M_{n} \DO
\BEGIN
	\FOR k \GETS 1 \TO M_{n} \DO
		\BEGIN
			\IF p_{in}^{(n)}<0 \AND p_{kn}^{(n)}>0 \THEN
				\BEGIN
					\FOR j \GETS 1 \TO n-1 \ADO p_{M_{n}(i-1)+k,j}^{(n-1)} \GETS -\frac{p_{ij}^{(n)}p_{kn}^{(n)}-p_{kj}^{(n)}p_{in}^{(n)}}{p_{in}^{(n)}-p_{kn}^{(n)}}
					\\
					a_{M_{n}(i-1)+k}^{(n-1)} \GETS (a_{i}^{(n)})^{-\frac{p_{kn}^{(n)}}{p_{in}^{(n)}-p_{kn}^{(n)}}}(a_{k}^{(n)})^{\frac{p_{in}^{(n)}}{p_{in}^{(n)}-p_{kn}^{(n)}}}
				\END
			\ELSE
				\BEGIN
					\FOR j \GETS 1 \TO n-1 \ADO p_{M_{n}(i-1)+k,j}^{(n-1)} \GETS 0
					\\
					a_{M_{n}(i-1)+k}^{(n-1)} \GETS \mathbb{0}
				\END
		\END
		\\
		\FOR j \GETS 1 \TO n-1 \ADO p_{M_{n}^{2}+i,j}^{(n-1)} \GETS p_{ij}^{(n)}
	\\
	a_{M_{n}^{2}+i}^{(n-1)} \GETS (h_{n}^{-p_{in}^{(n)}}\oplus g_{n}^{-p_{in}^{(n)}})^{-1}a_{i}^{(n)}
\END
\ENDPROCEDURE
\PROCEDURE{Constraints}{\bm{R}_{n-1},\bm{S}_{n-1},\bm{c}_{n-1},\bm{d}_{n-1},\bm{P}_{n},\bm{a}_{n},M_{n},n,\mu}
\COMMENT{Construct $\bm{R}_{n-1}$, $\bm{S}_{n-1}$, $\bm{c}_{n-1}$ and $\bm{d}_{n-1}$  from $\bm{P}_{n}$, $\bm{a}_{n}$ and $\mu$}
\\
\FOR i \GETS 1 \TO M_{n} \DO
	\IF p_{in}^{(n)}<0 \CTHEN
		\BEGIN
			\FOR j \GETS 1 \TO n-1 \ADO r_{ij}^{(n-1)} \GETS p_{ij}^{(n)}/p_{in}^{(n)}
			\\
			c_{i}^{(n-1)} \GETS \mu^{1/p_{in}^{(n)}}(a_{i}^{(n)})^{-1/p_{in}^{(n)}}
		\END
	\ELSEIF p_{in}^{(n)}>0 \CTHEN
		\BEGIN
			\FOR j \GETS 1 \TO n-1 \ADO s_{ij}^{(n-1)} \GETS p_{ij}^{(n)}/p_{in}^{(n)}
			\\
			d_{i}^{(n-1)} \GETS \mu^{-1/p_{in}^{(n)}}(a_{i}^{(n)})^{1/p_{in}^{(n)}}
		\END
\ENDPROCEDURE
\MAIN
\COMMENT{Construct arrays for forward elimination / backward substitution}
\\
\GLOBAL \bm{g},\bm{h}
\\
\bm{P}_{N} \GETS \bm{P},\ \bm{a}_{N} \GETS \bm{a},\ M_{n} \GETS M
\\
\FOR n \GETS N \DOWNTO 1 \ADO
	\BEGIN
		\CALL{Objectives}{\bm{P}_{n-1},\bm{a}_{n-1},\bm{P}_{n},\bm{a}_{n},M_{n},n}
		\\
		M_{n-1} \GETS M_{n}^{2}+M_{n}
	\END
	\\
	\mu \GETS \displaystyle{\bigoplus_{1\leq i\leq M_{0}}}a_{i}^{(0)}
\\	
\FOR n \GETS N \DOWNTO 1 \ADO \CALL{Constraints}{\bm{R}_{n-1},\bm{S}_{n-1},\bm{c}_{n-1},\bm{d}_{n-1},\bm{P}_{n},\bm{a}_{n},M_{n},n,\mu}	
\ENDMAIN	
\end{pseudocode}
\end{algorithm}

In conclusion, we observe that the procedure can be applied to solve those polynomial optimization problems, where the powers may have real exponents. In the case when the coefficients, bounds and powers in the problem are given by rational numbers, the procedure allows one to obtain exact solutions using rational arithmetic with an appropriate symbolic computation software.

\subsection{Computational Complexity}

Let us discuss the computational complexity of the solution procedure described above. Since the procedure shares a common core with the Fourier--Motzkin elimination for linear inequalities, it has similar computational performance and requirements (see, e.g.,~\cite{Ziegler1995Lectures,Schrijver1998Theory,Khachiyan2009Fourier} for complexity results for the Fourier--Motzkin elimination). First note that the most computationally intensive part of the procedure is the construction of new polynomials at each step of the backward elimination, which involves the calculation of the coefficient and exponents for each monomial. Therefore, as a rough measure of the time and space complexity, one can consider the number of monomials in all polynomial objective functions that appeared in the elimination process. As the number of monomials in the objective functions, in going from $n$ variables to $n-1$, formally increases as $M_{n-1}=M_{n}^{2}+M_{n}$, it has a quadratic growth rate, even though the actual number of new monomials is, in fact, not greater than $M_{n}^{2}/4+M_{n}$. 

As a result, the overall number of monomials in the objective functions obtained at $N$ elimination steps can be estimated as $O(M^{2^{N}})$ (see also~\cite{Krivulin2020Usingparameter}). This estimate shows polynomial growth with respect to $M$ and double exponential with $N$, which leads to excessively high computational requirements even for moderate numbers $M$ and $N$.

The rapid growth of the number of monomials in polynomials constructed at each step can be somewhat compensated for by decreasing the number of redundant monomials in the polynomials. The redundant monomials do not affect the value of the polynomial and can be removed to reduce the computational complexity of the procedure. An efficient reduction scheme to decrease the double-exponential growth in complexity of the Fourier--Motzkin elimination is known as the double description method, which eliminates redundancy to lower the complexity to exponential (see, e.g.,~\cite{Ziegler1995Lectures,Schrijver1998Theory,Khachiyan2009Fourier} for further details and references). Due to a close correspondence between the Fourier--Motzkin elimination and the proposed procedure, one can extend the double description method to decrease the order of computational complexity of the procedure. We do not focus here on the details of the adaptation of the double description method, but only comment on its applicability to the procedure under development and suggest the adaptation of the method for further~research.

\section{Application Examples}
\label{S-AE}

As an example of application of the proposed technique, consider a discrete linear Chebyshev approximation problem, which finds wide application including the solution of overdetermined systems of linear equations in computational algebra, the least maximum absolute deviation estimation in statistics and others. The problem is formulated in terms of conventional mathematics as follows (see, e.g.,~\cite{Krivulin2020Usingparameter}). Given $X_{ij},Y_{i},g_{j},h_{j}\in\mathbb{R}$ for all $i=1,\ldots,K$ and $j=1,\ldots,N$, the problem is to find the unknown parameters $\theta_{j}\in\mathbb{R}$ for all $j=1,\ldots,N$ that yield
\begin{equation}
\begin{aligned}
&&
\min_{\theta_{1},\ldots,\theta_{N}}
&&&
\max_{1\leq i\leq K}
\left|
\sum_{j=1}^{N}
X_{ij}\theta_{j}-Y_{i}
\right|;
\\
&&
\text{\upshape s.t.}
&&&
g_{j}
\leq
\theta_{j}
\leq
h_{j},
\qquad
j=1,\ldots,N.
\end{aligned}
\label{P-mintheta1thetaN}
\end{equation}

We observe that one can rearrange this approximation problem into a linear program and then numerically solve this program by using, for example, the simplex algorithm with exponential time or the Karmarkar algorithm with polynomial time.

Let us verify that problem~\eqref{P-mintheta1thetaN} can be represented in the form of the polynomial optimization problem at~\eqref{P-minx1xNaix1pi1xNpiN-gjleqxjleqhj}. First, we consider the expression under the maximum sign and rewrite it as
\begin{equation*}
\left|
\sum_{j=1}^{N}
X_{ij}\theta_{j}-Y_{i}
\right|
=
\max(Y_{i}-X_{i1}\theta_{1}-\cdots-X_{iN}\theta_{N},
-Y_{i}+X_{i1}\theta_{1}+\cdots+X_{iN}\theta_{N}).
\end{equation*}

Furthermore, we note that in terms of the max-plus algebra (the semifield $\mathbb{R}_{\max,+}$), the maximum on the right-hand side becomes the tropical sum of two monomials
\begin{equation*}
Y_{i}\theta_{1}^{-X_{i1}}\cdots\theta_{N}^{-X_{iN}}
\oplus
Y_{i}^{-1}\theta_{1}^{X_{i1}}\cdots\theta_{N}^{X_{iN}}.
\end{equation*}

To represent the objective function in~\eqref{P-mintheta1thetaN}, which is defined as the maximum of the above tropical sums, we put $M=2K$ and then change variables for all $i=1,\ldots,M$ and $j=1,\ldots,N$ by setting  
\begin{equation*}
x_{j}=\theta_{j},
\qquad
p_{ij}
=
\begin{cases}
-X_{ij},
&
\text{if $i\leq K$};
\\
X_{i-K,j},
&
\text{if $i>K$};
\end{cases}
\qquad
a_{i}
=
\begin{cases}
Y_{i},
&
\text{if $i\leq K$};
\\
-Y_{i-K},
&
\text{if $i>K$}.
\end{cases}
\end{equation*}

With the variables $g_{j}$ and $h_{j}$ left unchanged, we rewrite problem~\eqref{P-mintheta1thetaN} in terms of $\mathbb{R}_{\max,+}$ by using the new variables, which gives a problem in the form of~\eqref{P-minx1xNaix1pi1xNpiN-gjleqxjleqhj}. Then, the application of the computational procedure described above allows the Chebyshev approximation problem to be solved in a finite number of steps.

We now present numerical results of the exact solution of example problems of the Chebyshev approximation. We implemented a software code developed to perform all steps of the procedure with MATLAB rel.~R2021b by using symbolic computations in rational arithmetic. The numerical experiments were run on a desktop computer equipped with a 3.40 GHz Intel Xeon E3-1231 v3 CPU, 32GB DDR3 RAM and 64-bit Windows 10 Enterprise OS.

\begin{example}
We start with the solution of a linear Chebyshev approximation problem where one needs to fit three unknown parameters $\theta_{1},\theta_{2},\theta_{3}$ to achieve
\begin{equation*}
\begin{aligned}
&&
\min_{\theta_{1},\theta_{2},\theta_{3}}
&&&
\max(
|3\theta_{1}-\theta_{2}+2\theta_{3}-2|,
|\theta_{1}-2\theta_{2}+\theta_{3}-1|,
|2\theta_{1}+3\theta_{2}-\theta_{3}+1|,
|4\theta_{2}-\theta_{3}|
);
\\
&&
\text{\upshape s.t.}
&&&
0
\leq
\theta_{1}
\leq
1,
\quad
0
\leq
\theta_{2}
\leq
1,
\quad
0
\leq
\theta_{3}
\leq
1.
\end{aligned}
\end{equation*}

After changing the variables in this problem and rewriting in terms of the semifield $\mathbb{R}_{\max,+}$, we arrive at the tropical polynomial optimization problem with $M=8$ and $N=3$ to find
\begin{equation*}
\begin{aligned}
&&
\min_{x_{1},x_{2},x_{3}}
&&&
2^{-1}x_{1}^{3}x_{2}^{-1}x_{3}^{2}
\oplus
2x_{1}^{-3}x_{2}x_{3}^{-2}
\oplus
1^{-1}x_{1}x_{2}^{-2}x_{3}
\oplus
1x_{1}^{-1}x_{2}^{2}x_{3}^{-1}
\\
&&&&&
\vphantom{\min_{x_{1},x_{2},x_{3}}}
\phantom{2^{-1}x_{1}^{3}x_{2}^{-1}x_{3}^{2}\oplus2x_{1}^{-3}x_{2}x_{3}^{-2}}
\oplus
1x_{1}^{-2}x_{2}^{-3}x_{3}^{-1}
\oplus
1^{-1}x_{1}^{2}x_{2}^{3}x_{3}
\oplus
x_{2}^{4}x_{3}^{-1}
\oplus
x_{2}^{-4}x_{3}^{1};
\\
&&
\text{\upshape s.t.}
&&&
0
\leq
x_{1}
\leq
1,
\quad
0
\leq
x_{2}
\leq
1,
\quad
0
\leq
x_{3}
\leq
1.
\end{aligned}
\end{equation*}

The application of the solution procedure results, in the setting of $\mathbb{R}_{\max,+}$, in the minimum of the objective function $\mu=3^{1/7}$ and a single solution in the form $x_{1}=0$, $x_{2}=1^{1/7}$, $x_{3}=1$. In terms of the initial approximation problem, these results correspond to the minimum approximation error equal to $3/7$ and the parameter estimates given by
\begin{equation*}
\theta_{1}
=0,
\qquad
\theta_{2}
=1/7,
\qquad
\theta_{3}
=1.
\end{equation*}

To obtain this exact rational solution, the procedure takes about 26~sec of computer time.
\end{example}

\begin{example}
We now consider a Chebyshev approximation problem to find
\begin{equation*}
\begin{aligned}
&&
\min_{\theta_{1},\theta_{2},\theta_{3}}
&&&
\max(
|3\theta_{1}-\theta_{2}+2\theta_{3}+2|,
|\theta_{1}+2\theta_{2}-2\theta_{3}-1|,
|2\theta_{1}-3\theta_{2}+\theta_{3}+1|,
|2\theta_{2}-\theta_{3}|,
\\
&&&&&
\vphantom{\min_{\theta_{1},\theta_{2},\theta_{3}}}
\phantom{\max(\ }
|\theta_{1}+2\theta_{2}-\theta_{3}+1|,
|3\theta_{1}+\theta_{2}-1|,
|\theta_{1}+\theta_{2}-\theta_{3}+2|,
|\theta_{1}+\theta_{2}+2\theta_{3}|,
\\
&&&&&
\vphantom{\min_{\theta_{1},\theta_{2},\theta_{3}}}
\phantom{\max(|\theta_{1}+2\theta_{2}-\theta_{3}+1|,|3\theta_{1}+\theta_{2}-1|,\ }
|3\theta_{2}+\theta_{3}+1|,
|2\theta_{1}+\theta_{2}+3|
);
\\
&&
\text{\upshape s.t.}
&&&
-1/4
\leq
\theta_{1}
\leq
1/4,
\quad
-1/4
\leq
\theta_{2}
\leq
1/4,
\quad
-1/4
\leq
\theta_{3}
\leq
1/4.
\end{aligned}
\end{equation*}

A direct translation into the language of max-plus algebra leads to a tropical optimization problem with $M=20$ and $N=3$ in the form
\begin{equation*}
\begin{aligned}
&&
\min_{x_{1},x_{2},x_{3}}
&&&
2x_{1}^{3}x_{2}^{-1}x_{3}^{2}
\oplus
2^{-1}x_{1}^{-3}x_{2}x_{3}^{-2}
\oplus
1^{-1}x_{1}x_{2}^{2}x_{3}^{-2}
\oplus
1x_{1}^{-1}x_{2}^{-2}x_{3}^{2}
\\
&&&&&
\vphantom{\min_{x_{1},x_{2},x_{3}}}
\oplus
1x_{1}^{2}x_{2}^{-3}x_{3}
\oplus
1^{-1}x_{1}^{-2}x_{2}^{3}x_{3}^{-1}
\oplus
x_{2}^{2}x_{3}^{-1}
\oplus
x_{2}^{-2}x_{3}^{1}
\oplus
1x_{1}x_{2}^{2}x_{3}^{-1}
\oplus
1^{-1}x_{1}^{-1}x_{2}^{-2}x_{3}
\\
&&&&&
\vphantom{\min_{x_{1},x_{2},x_{3}}}
\oplus
1^{-1}x_{1}^{3}x_{2}
\oplus
1x_{1}^{-3}x_{2}^{-1}
\oplus
2x_{1}x_{2}x_{3}^{-1}
\oplus
2^{-1}x_{1}^{-1}x_{2}^{-1}x_{3}
\oplus
x_{1}x_{2}x_{3}^{2}
\oplus
x_{1}^{-1}x_{2}^{-1}x_{3}^{-2}
\\
&&&&&
\vphantom{\min_{x_{1},x_{2},x_{3}}}
\phantom{\oplus1^{-1}x_{1}^{3}x_{2}\oplus1x_{1}^{-3}x_{2}^{-1}}
\oplus
1x_{2}^{3}x_{3}
\oplus
1^{-1}x_{2}^{-3}x_{3}^{-1}
\oplus
3x_{1}^{2}x_{2}
\oplus
3^{-1}x_{1}^{-2}x_{2}^{-1};
\\
&&
\text{\upshape s.t.}
&&&
1^{-1/4}
\leq
x_{1}
\leq
1^{1/4},
\quad
1^{-1/4}
\leq
x_{2}
\leq
1^{1/4},
\quad
1^{-1/4}
\leq
x_{3}
\leq
1^{1/4}.
\end{aligned}
\end{equation*}

Solving the problem yields the minimum $\mu=13^{1/8}$ attained at $x_{1}=1^{-1/4}$, $x_{2}=1^{1/8}$, $x_{3}=1^{1/4}$. Turning back to the Chebyshev approximation problem in the conventional algebra setting, we obtain the minimum approximation error $13/8$ and the parameter estimates 
\begin{equation*}
\theta_{1}
=-1/4,
\qquad
\theta_{2}
=1/8,
\qquad
\theta_{3}
=1/4.
\end{equation*}

The solution of this problem using symbolic computations in rational arithmetic requires about 107~min of computer time. The sharp increase in the execution time of this task compared to the task in the first example reflects the double-exponential growth of the number of monomials in the objective functions obtained in the variable elimination steps, and thus indicates the importance of further research on the improvement of the efficiency of the method.
\end{example}

\section{Conclusions}
\label{S-C}

We considered optimization problems, where the objective function to be minimized is defined in terms of tropical (idempotent) algebra as a multivariate polynomial with rational exponents (a tropical Puiseux polynomial), subjected to box constraints. We have proposed a solution procedure that uses variable elimination to solve the problem in a finite number of iterations. The procedure involves two phases: forward elimination and backward substitution of variables, and can be considered as an extension of the Fourier--Motzkin elimination method for systems of linear inequalities in ordered fields to solving polynomial optimization problems in ordered tropical semifields.

The proposed procedure offers a means for both symbolic computations to obtain exact solutions using rational arithmetic and numerical computations to find approximate solutions using floating-point calculations. When solving real-world problems represented in terms of tropical polynomial optimization, the procedure can serve to supplement and complement existing solutions based on conventional mathematics.

The procedure is scalable in the sense that it potentially allows for solving polynomial optimization problems of any size by using the same computational algorithm and formulas. However, the computational complexity of the procedure, as for the Fourier--Motzkin elimination, grows very fast with increasing of the number of monomials in the objective function. Therefore, an extension of the double description method to improve the computational complexity of the procedure is of prime interest for future research. Further development of the procedure to take into account other types of constraints presents another promising line of investigation.

\section*{Acknowledgments}
The author is very grateful to three anonymous referees for their valuable comments and suggestions.

\bibliographystyle{abbrvurl}
\bibliography{Algebraic_solution_of_tropical_polynomial_optimization_problems}

\end{document}